\documentclass[12pt]{amsart}
\usepackage{amsmath,amsthm,amsfonts,amssymb,latexsym,enumerate,xcolor}
\usepackage{showlabels}
\usepackage[colorlinks,pagebackref]{hyperref}
\allowdisplaybreaks

\newcommand{\RR}{{\mathbb{R}}}

\newcommand{\ZZ}{{\mathbb{Z}}}

\newcommand{\bF}{{\mathbf F}}

\newcommand{\bC}{{\mathbf C}}

\newcommand{\bG} {\mathbf G}

\newcommand{\bZ}{{\mathbf Z}}

\newcommand{\EC} {\mathcal E}

\newcommand{\SSS}{\mathsf{S}}
\newcommand{\AAA}{\mathsf{A}}

\newcommand{\DDD}{\mathsf{D}}

\newcommand{\cl}{{{\operatorname{cl}}}}

\newcommand{\Irr}{{{\operatorname{Irr}}}}
\newcommand{\Van}{{{\operatorname{Van}}}}
\newcommand{\GL}{\operatorname{GL}}

\newcommand{\PSL}{\operatorname{PSL}}

\newcommand{\SL}{\operatorname{SL}}

\newcommand{\St}{{\mathsf {St}}}

\let\la=\lambda
\let\si=\sigma
\let\ga=\gamma

\newtheorem{thm}{Theorem}[section]
\newtheorem{lem}[thm]{Lemma}

\newtheorem{pro}[thm]{Proposition}

\newtheorem{que}[thm]{Question}

\newtheorem{thml}{Theorem}

\newtheorem{conl}[thml]{Conjecture}

\theoremstyle{definition}
\newtheorem{rem}{Remark}

\numberwithin{equation}{section}

\newcommand{\ua}{{\uparrow}}

\def\Md#1{\text{ }(\text{\rm mod } #1)\,}

\newcommand{\sgn}{\mathrm{sgn}}
\newcommand{\bfd}{\mathbf{d}}

\begin{document}

\title[Common zeros of irreducible characters]
{Common zeros of irreducible characters}

\author[N.\,N. Hung]{Nguyen N. Hung}
\address{Department of Mathematics, The University of Akron, Akron,
OH 44325, USA} \email{hungnguyen@uakron.edu}

\author[A. Moret\'{o}]{Alexander Moret\'{o}}
\address{Departamento de Matem\'aticas, Universidad de Valencia,
46100 Burjassot, Valencia, Spain} \email{alexander.moreto@uv.es}

\author[L. Morotti]{Lucia Morotti}
\address{Mathematisches Institut, Heinrich-Heine-Universit\"{a}t D\"{u}sseldorf,
40225 D\"{u}sseldorf, Germany} \email{lucia.morotti@uni-duesseldorf.de}

\thanks{The second author thanks Noelia Rizo for helpful conversations.
The first author is grateful for the support of an UA
Faculty Research Grant. The research of the second author is
supported by Ministerio de Ciencia e Innovaci\'on (Grant
PID2019-103854GB-I00 funded by MCIN/AEI/ 10.13039/501100011033) and
Generalitat Valenciana CIAICO/2021/163.}




\begin{abstract}
We study the zero-sharing behavior among irreducible characters of a
finite group. For symmetric groups $\SSS_n$, it is proved that, with
one exception, any two irreducible characters have at least one
common zero. To further explore this phenomenon, we introduce
\emph{the common-zero graph} of a finite group $G$, with non-linear
irreducible characters of $G$ as vertices, and edges connecting
characters that vanish on some common group element. We show that
for solvable and simple groups, the number of connected components
of this graph is bounded above by 3. Lastly, the result for $\SSS_n$
is applied to prove the non-equivalence of the metrics on
permutations induced from faithful irreducible characters of the
group.
\end{abstract}

\maketitle



\section{Introduction}

Studying zeros of characters (and character values in general) is a
fundamental problem in the representation theory of finite groups.
While values of linear characters are never zero, it is a classical
result of Burnside that every non-linear irreducible character
always vanishes at some element. This was improved by Malle,
Navarro, and Olsson \cite{mno} who showed that the element can be
chosen to be of prime-power order.

More recently, Miller \cite{Miller} proved that almost all values of
irreducible characters of the symmetric group $\SSS_n$ are zero as
$n$ increases. Specifically, if $P_n$ is the probability that
$\chi(\sigma)=0$ where $\chi$ is uniformly chosen at random from the
set $\Irr(\SSS_n)$ of irreducible characters of $\SSS_n$ and
$\sigma$ is chosen at random from $\SSS_n$, then $P_n\rightarrow 1$
as $n\rightarrow \infty$. This remarkable observation has inspired
several subsequent results on the abundance of zeros of character
values, see \cite{Gallagher,Larsen-Miller,Ono,Peluse}.

Here we investigate the many-zero phenomenon from a different
perspective: \emph{when do two distinct irreducible characters of a
finite group have a common zero}? For symmetric groups, our answer
is complete.

\begin{thml}\label{main:A}
Let $n\in\ZZ^{\geq 8}$ and $\chi,\psi\in\Irr(\SSS_n)$ both have
degree larger than 1. Then $\chi$ and $\psi$ have a common zero if
and only if $\{\chi(1),\psi(1)\}\neq \{n(n-3)/2,(n-1)(n-2)/2\}$.
\end{thml}

\begin{rem}
It is well-known that irreducible characters of $\SSS_n$ are
parameterized by partitions of $n$, and so we use $\chi^\la$ to
denote the character corresponding to a partition $\lambda$. The
characters of degrees $n(n-3)/2$ and $(n-1)(n-2)/2$ are
corresponding to the partitions $(n-2,2)$, $(n-2,1^2)$ and their
conjugates. Since $\chi^{(n-2,2)}\equiv\chi^{(n-2,1^2)}+1\Md{2}$
(see \cite[p. 93]{JamesBook}), $\chi^{(n-2,2)}$ and
$\chi^{(n-2,1^2)}$ indeed do not share common zeros.
\end{rem}

In Section \ref{sec:metric}, we discuss an application of Theorem
\ref{main:A} to a problem concerning the partition equivalence of
character-induced metrics on permutations. A metric on a finite
group $G$ is a binary function $d:G\times G\rightarrow \RR^{\geq 0}$
that assigns a non-negative real number to each pair of group
elements, satisfying the properties of a metric, such as positivity,
symmetry, and the triangle inequality. If $\chi$ is a faithful
character of $G$, then the induced function $\bfd_\chi$ defined by
$\bfd_\chi(a,b):=(\chi(1)- Re (\chi(ab^{-1})))^{1/2}$ is a
$G$-invariant metric on $G$ (see \cite[Section 6D]{Diaconis}). (Here
$Re(z)$ is the real part of a complex number $z$.) Let
$\mathcal{P}(\bfd_\chi)$ be the partition of $G$ determined by the
equivalence relation: $a\thicksim b$ if and only if
$\bfd_\chi(1,a)=\bfd_\chi(1,b)$. Two metrics $\bfd_\chi$ and
$\bfd_\psi$ are called partition equivalent (or
$\mathcal{P}$-equivalent for short) if
$\mathcal{P}(\bfd_\chi)=\mathcal{P}(\bfd_\psi)$ (see \cite{PV23}).
Using Theorem~\ref{main:A}, we prove that the metrics $\bfd_{\chi}$
on permutations induced from the faithful irreducible characters
$\chi$ of the group are pairwise non-$\mathcal{P}$-equivalent.

Zero-sharing behavior for arbitrary groups is harder to understand.
To explore it further, we introduce \emph{the common-zero graph} of
$G$, denoted by $\Gamma_v(G)$: the vertices are the non-linear
irreducible characters of $G$ and two characters are joined by an
edge if they vanish on some common element. Theorem~\ref{main:A}
shows that $\Gamma_v(\SSS_n)$ with $n\geq 8$ is almost complete and,
in particular, connected. This is not true in general. For solvable
groups, we prove

\begin{thml}\label{main:C}
Let $G$ be a finite solvable group. Then the number of connected
components of $\Gamma_v(G)$ is at most $2$.
\end{thml}


We have observed that two irreducible characters tend to have a
common zero if their degrees are not coprime. Although exceptions
exist, they are rare and difficult to find. This suggests that the
well-studied \emph{common-divisor graph} $\Gamma(G)$, whose vertices
are the same non-linear irreducible characters of $G$ and two
vertices are joined if their degrees are not coprime, is somewhat
close to a subgraph of $\Gamma_v(G)$. Manz, Staszewski, and Willems
\cite{msw} proved that $\Gamma(G)$ in fact has at most three
connected components for all groups. This observation together with
Theorems \ref{main:A} and \ref{main:C} suggest the following.

\begin{conl}\label{conj:main} The number
of connected components in $\Gamma_v(G)$  for any finite group $G$
is at most $3$.
\end{conl}

The next result provides further evidence for the conjecture.

\begin{thml}\label{main:D}
Let $G$ be a finite simple group. Then the number of connected
components of $\Gamma_v(G)$ is at most $3$.
\end{thml}

The bounds in both Theorems \ref{main:C} and \ref{main:D} are best
possible, as shown by $\SSS_4, \GL_2(3)$ and $\PSL_2(q)$ for several
choices of $q$.


The paper is organized as follows. In Section~\ref{sec:classes}, we
present some preliminary results on the dual graph of $\Gamma_v(G)$
for vanishing conjugacy classes. Section~\ref{sec:symmetric}
contains the proofs of Theorem~\ref{main:A} and the
alternating-group case of Theorem~\ref{main:D}. Theorem~\ref{main:C}
on solvable groups is proved in Section~\ref{sec:solvable}. We then
handle simple groups of Lie type and complete the proof of
Theorem~\ref{main:D} in Section~\ref{sec:Lie type}. The topic of
character-induced metrics on permutations is presented in
Section~\ref{sec:metric}. In the final
Section~\ref{section:relation-character-degree}, we discuss some
relationship between common zeros and character degrees.


\section{The dual graph for vanishing
classes}\label{sec:classes}

For studying the common-zero graph $\Gamma_v(G)$, it is sometimes
convenient to consider the dual one for \emph{vanishing conjugacy
classes}.

As usual, let $\Irr(G)$ denote the set of all (ordinary) irreducible
characters of a finite group $G$ and $\cl(G)$ the set of all
conjugacy classes. An element $x\in G$ is called vanishing if there
exists $\chi\in\Irr(G)$ such that $\chi(x)=0$. Accordingly, a
conjugacy class $K\in \cl(G)$ is called vanishing if it contains a
vanishing element. These concepts were introduced by Isaacs, Navarro
and Wolf in \cite{inw} and since then, they have been studied in
great depth. See the survey paper \cite{dps18} of Dolfi, Pacifici,
and Sanus for more information.

The dual graph we referred to, which we will denote by
$\Delta_v(G)$, has vertices being the vanishing conjugacy classes of
$G$ and two classes are joined if there exists an irreducible
character in $\Irr(G)$ that vanishes simultaneously on both classes.
The common-zero graph $\Gamma_v(G)$ and this $\Delta_v(G)$ have the
same number of connected components, among other things.

\begin{lem}
\label{rel}
Let $G$ be a finite group. Then:
\begin{enumerate}
\item
If  $\mathcal{C}$ is the set of conjugacy classes in one connected
component of $\Delta_v(G)$, then $\{\chi\in\Irr(G)\mid\text{\,
$\chi$ vanishes at some class in $\mathcal{C}$}\}$ is a connected
component in $\Gamma_v(G)$.
\item
If $\mathcal{A}$ is the set of characters in a connected component
in $\Gamma_v(G)$, then $\{K\in\cl(G)\mid\text{\,$\chi(K)=0$ for some
$\chi\in\mathcal{D}$}\}$ is a connected component in $\Delta_v(G)$.
\item
If $f$ is the map from the set of connected components of
$\Delta_v(G)$ to the set of connected components of $\Gamma_v(G)$
defined in (i) and $h$ is the map from the set of connected
components of $\Gamma_v(G)$ to the set of connected components of
$\Delta_v(G)$ defined in (ii), then $f$ and $h$ are bijections and
one is the inverse of the other. In particular, $\Gamma_v(G)$ and
$\Delta_v(G)$ have the same number of connected components.
\item
The difference between the diameter of a connected component
$\mathcal{C}$ of $\Delta_v(G)$ and $f(\mathcal{C})$ is at most one.
\end{enumerate}
\end{lem}

\begin{proof}
Let $\mathcal{A}=\{\chi\in\Irr(G)\mid\text{\, $\chi$ vanishes at
some class in $\mathcal{C}$}\}$. We want to see that if
$\alpha,\beta\in\mathcal{A}$, then there exists a path in
$\Gamma_v(G)$ joining $\alpha$ and $\beta$. Let $C,D\in\mathcal{C}$
such that $\alpha(C)=\beta(D)=0$. Since $\mathcal{C}$ is a connected
component in $\Delta_v(G)$, there exists a path
$$
C=C_0\leftrightarrow C_1\leftrightarrow\cdots\leftrightarrow C_{d-1}\leftrightarrow C_d=D
$$
joining $C$ and $D$. By the definition of the graph $\Delta_v(G)$,
this means that there exist $\chi_i\in\Irr(G)$ such that
$\chi_i(C_{i-1})=\chi_i(C_i)=0$ for $i=1,\dots,d$. Therefore,
$$
\alpha\leftrightarrow\chi_1\leftrightarrow\cdots\leftrightarrow\chi_d\leftrightarrow\beta
$$
is a path in $\Gamma_v(G)$ joining $\alpha$ and $\beta$. This proves
that $\mathcal{A}$ is contained in a connected component of
$\Gamma_v(G)$.

Similarly, we can prove that if $\mathcal{A}$ is the set of
characters in a connected component in $\Gamma_v(G)$, then
$\mathcal{D}=\{K\in\cl(G)\mid\text{\,$\chi(K)=0$ for some
$\chi\in\mathcal{A}$}\}$ is contained in a connected component in
$\Delta_v(G)$.

Now, suppose that $\mathcal{A}\subsetneq\mathcal{B}$, where
$\mathcal{B}$ is a connected component of $\Gamma_v(G)$. Let
$\gamma\in\mathcal{B}-\mathcal{A}$. Therefore, $\gamma(K)\neq0$ for
every $K\in\mathcal{C}$. Since $\alpha,\gamma\in\mathcal{A}$, there
exists a path in $\Gamma_v(G)$ joining $\alpha$ and $\gamma$.  These
means that there exist $\eta,\mu$ in this path that are linked, one
of them belongs to $\mathcal{A}$ and the other does not. Suppose
that $\eta\in\mathcal{A}$ and $\mu\not\in\mathcal{A}$. Therefore,
there exists $D\not\in\mathcal{C}$ such that $\eta(D)=\mu(D)=0$.
Since $\eta\in\mathcal{A}$, there exists $C\in\mathcal{C}$ such that
$\eta(C)=0$. By the definition of $\Delta_v(G)$, $C$ and $D$ are
linked by means of $\eta$. This is a contradiction.

We omit the proof of the remaining parts, which are similar.
\end{proof}


The next result is essential in studying the number of connected
components of our graphs. We write $\Van(\chi)$ for the set of zeros
(or roots) of a character $\chi$. Clearly two characters $\chi$ and
$\psi$ have common zeros if and only if $\Van(\chi)\cap \Van(\psi)
\neq \emptyset$. This explains the relevance of $\Van(\chi)$ in the
current context.

\begin{lem}
\label{res} Let $G$ be a group, $N\trianglelefteq\, G$ and
$\chi\in\Irr(G)$. If $\Van(\chi)\subseteq N$, then
$\chi_N\in\Irr(N)$.
\end{lem}

\begin{proof}
Let $\theta\in\Irr(N)$ be lying under $\chi$ and let $T=I_G(\theta)$.
By Clifford's correspondence, there exists $\psi\in\Irr(T)$ such that $\psi^G=\chi$.
By the formula for the induced character, $\chi$ vanishes on $G-\bigcup_{g\in G}T^g$.
Since a group is never the union of the conjugates of a proper subgroup, this implies that
$\theta$ is $G$-invariant. Thus $\chi_N=e\theta$
for some positive integer $e$
and we want to see that $e=1$. Arguing by contradiction, assume that $e>1$.

 Let $(G^*,N^*,\theta^*)$ be a character triple isomorphic to
$(G,N,\theta)$ with $N^*\leq\bZ(G^*)$ and $\theta^*$ linear and
faithful. (We refer the reader to Section 5.4 and in particular
Corollary 5.9 of \cite{nav} for background on character triples.)
Let $\chi^\ast\in\Irr(G^\ast\mid \theta^\ast)$ be corresponding to
$\chi$ under the isomorphism. Note that $\chi^*$ is not linear and
$\chi^*_{N^*}=e\theta^*$. Since
$\theta^*$ is linear, $\Van(\chi^*)\cap N^*=\emptyset$. Therefore,
there exists $x^*\in G^*-N^*$ such that $\chi^*(x^*)=0$. Let $x\in
G$ such that \[(Nx)^\ast=N^\ast x^\ast,\] where
$^\ast:G/N\longrightarrow G^\ast/N^\ast$ is the associated group
isomorphism. Note that $x\not\in N$. By Lemma 5.17(a) of \cite{nav},
there exists an algebraic integer $\alpha$ such that
$$\chi(x)=\alpha\chi^\ast(x^\ast)=0,$$
contradicting the hypothesis that $\Van(\chi)\subseteq N$.
\end{proof}

Let $\Van(G)$ denote the set of vanishing elements of $G$, so
$\Van(G)=\bigcup_{\chi\in\Irr(G)}\Van(\chi)$. Although $\Van(G)$ has
been extensively studied in the literature, it is not fully understood
yet. For instance, for solvable groups it is an open conjecture
\cite{inw} that $\Van(G)$ always contains $G-\bF(G)$. The case of
nilpotent groups is known though.

\begin{lem}
\label{vn}
Let $G$ be a nilpotent group. Then $\Van(G)=G-\bZ(G)$.
\end{lem}

\begin{proof}
This is \cite[Theorem B]{inw}.
\end{proof}

As we will see throughout, and in particular in Section
\ref{section:relation-character-degree}, there is a somewhat
mysterious relationship between the set of zeros of an irreducible
character and its degree. There are two graphs associated to
character degrees that have been well studied in the literature. The
first is the already mentioned \emph{common-divisor graph}
$\Gamma(G)$. (We note that, in the definition of $\Gamma(G)$, one
could take the character degrees of $G$ instead to be the vertices.
The resulting graph and $\Gamma(G)$ are fundamentally the same.) The
second is the \emph{prime graph} $\Delta(G)$ with vertices being the
primes that divide some character degree of $G$ and two vertices are
joined if the product of the primes divides the degree of some
irreducible character of $G$. The survey paper \cite{lew} of Lewis
is a good place for an overview of the known results on these and
other character-degree related graphs, up until 2008. Many more have
been obtained since then.

We already mentioned that two irreducible characters tend to have a
common zero if their degrees are not coprime. Therefore, $\Gamma(G)$
is often close to a subgraph of $\Gamma_v(G)$, but not always a
subgraph. There are counterexamples among groups of Lie type, like
$\PSL_2(11)$, and among sporadic groups, like $M_{12}$. There are
also solvable counterexamples, the smallest one being ${\tt
SmallGroup}(324, 160)$ in the notation of GAP \cite{gap}. It would
be interesting to understand when $\Gamma(G)$ is a subgraph of
$\Gamma_v(G)$. If $\Gamma(G)$ is a subgraph of $\Gamma_v(G)$, then
the number of connected components of $\Gamma_v(G)$ is at most 3, by
the main result of \cite{msw}.


\section{Alternating and symmetric groups}\label{sec:symmetric}

In this section we prove Theorem \ref{main:A} and the alternating
group case of Theorem~\ref{main:D}. We start by comparing zeros of
irreducible characters of symmetric groups and components of their
restrictions to alternating groups.

\begin{lem}\label{L130123}
Let $\chi\in\Irr(\SSS_n)$ and $\psi\in\Irr(\AAA_n)$ be such that
$\psi$ appears in $\chi_{\AAA_n}$ and let
$\si\in\AAA_n$. Then $\chi(\si)=0$ if and only if $\psi(\si)=0$.
\end{lem}

\begin{proof}
If $\psi(\si)\not\in\{\chi(\si),\chi(\si)/2\}$, $\chi$ is labeled by
the partition $\la$ and $\mu$ is the cycle partition of $\si$, then
by \cite[Theorems 2.5.7 and 2.5.13]{jk} $\la$ is self-conjugated and $\mu$ is the
partition consisting of the diagonal hook-lengths of $\la$. But in
this case $\chi(\si)\not=0$ by \cite[Corollary 2.4.8]{jk} and
$\psi(\si)\not=0$ by \cite[Theorem 2.5.13]{jk}.
\end{proof}

In the next theorem we show that almost always irreducible characters of $\SSS_n$ have a common zero in $\AAA_n$. It can be checked that the following result is false for small $n$.

\begin{thm}\label{L160123}
Let $n\geq 8$ and $\chi,\psi\in\Irr(\SSS_n)$ both have degree larger than 1. Then
$\chi$ and $\psi$ have no common zero if and only if,
up to multiplying $\chi$ or $\psi$
with $\sgn$, $\{\chi,\psi\}=\{\chi^{(n-2,2)},\chi^{(n-2,1^2)}\}$. If $n\geq 9$ and $\chi$ and $\psi$ have a common zero then they also have a common zero in $\AAA_n$.
\end{thm}

%
%

\begin{proof}
If $\chi=\chi^{(n-2,2)}$ and $\psi=\chi^{(n-2,1^2)}$ (up to exchange or multiplying with $\sgn$) then by
\cite[p. 93]{JamesBook} we have that $\psi\equiv\chi+1\Md{2}$. So in this
case $\chi$ and $\psi$ cannot have common zeros. So we may now
assume that this is not the case.

The cases $n=8$ and $9$ can be checked looking at character tables. So assume now that $n\geq 10$.

Let $\la$ and $\mu$ be the partitions labeling $\chi$ and $\psi$
respectively, so that $\chi=\chi^\la$ and $\psi=\chi^\mu$. Note that
$\la,\mu\not\in\{(n),(1^n)\}$ since $\chi$ and $\psi$ have degree
larger than 1. Further for any partition $\ga\vdash n$ let
$\tau_\ga\in\SSS_n$ have cycle partition $\ga$.

\medskip

{\bf Case 1:} $n$ is even. We have that $\tau_{(n-k,k)}\in\AAA_n$
for any $1\leq k\leq n/2$. For $1\leq k\leq 4$ let
$N_{n,k}:=\{\ga\vdash n|\chi^\ga(\tau_{(n-k,k)})=0\}$ be the set of
partitions labeling characters which do not vanish on
$\tau_{(n-k,k)}$. We then have that $N_{n,k}=B_{n,k}\cup C_{n,k}\cup
D_{n,k}$ with
\begin{align*}
B_{n,1}=&\{(n),(1^n)\},\\
C_{n,1}=&\emptyset,\\
D_{n,1}=&\{(n-h,2,1^{h-2})|2\leq h\leq n-2\},\\
B_{n,2}=&\{(n),(n-1,1),(2,1^{n-2}),(1^n)\},\\
C_{n,2}=&\{(n-2,2),(2^2,1^{n-4})\},\\
D_{n,2}=&\{(n-h,3,1^{h-3})|3\leq h\leq n-3\}\cup\{(n-h,2^2,1^{h-4})|4\leq h\leq n-2\},\\
B_{n,3}=&\{(n),(n-1,1),(n-2,1^2),(3,1^{n-3}),(2,1^{n-2}),(1^n)\},\\
C_{n,3}=&\{(n-3,3),(n-3,2,1),(n-4,2^2),(3^2,1^{n-6}),(3,2,1^{n-5}),(2^3,1^{n-6})\},\\
D_{n,3}=&\{(n-h,4,1^{h-4})|4\leq h\leq n-4\}\cup\{(n-h,3,2,1^{h-5})|5\leq h\leq n-3\}\\
&\cup\{(n-h,2^3,1^{h-6})|6\leq h\leq n-2\},\\
B_{n,4}=&\{(n),(n-1,1),(n-2,1^2),(n-3,1^3),(4,1^{n-4}),(3,1^{n-3}),(2,1^{n-2}),(1^n)\},\\
C_{n,4}=&\{(n-4,4),(n-4,3,1),(n-4,2,1^2),(n-5,3,2),(n-5,2^2,1),(n-6,2^3),\\
&(4^2,1^{n-8}),(4,3,1^{n-7}),(4,2,1^{n-6}),(3^2,2,1^{n-8}),(3,2^2,1^{n-7}),(2^4,1^{n-8})\},\\
D_{n,4}=&\{(n-h,5,1^{h-5})|5\leq h\leq n-5\}\cup\{(n-h,4,2,1^{h-6})|6\leq h\leq n-4\}\\
&\cup\{(n-h,3,2^2,1^{h-7})|7\leq h\leq n-3\}\cup\{(n-h,2^4,1^{h-8})|8\leq h\leq n-2\}.
\end{align*}
To see this note that in view of the Murnaghan-Nakayama formula (see
\cite[\S 2.4.7]{jk}) if $\ga\in N_{n,k}$ then $\ga$ is obtained by
adding a $(n-k)$-hook to a hook-partition of $k$ and that this
condition is also sufficient for $\ga\in N_{n,k}$ since $k\leq
4<n/2$ (so that there is at most one way in which hooks of the given
length can be recursively removed from any given partition of $n$).

So if $\chi$ and $\psi$ have no common zero in $\AAA_n$ then
$\{\la,\mu\}\cap N_{n,k}\not=\emptyset$ for any $1\leq k\leq 4$.
Note that the sets $D_{n,k}$ are pairwise disjoint. The same is true
for the sets $C_{n,k}$ since $n\geq 10$. Assume first that neither
$\la$ nor $\mu$ is in
\[\cup B_{n,k}\setminus\{(n),(1^n)\}=\{(n-1,1),(n-2,1^2),(n-3,1^3),(4,1^{n-4}),(3,1^{n-3}),(2,1^{n-2})\}.\]
Then $\la\in C_{n,k_1}\cap D_{n,k_2}$ and $\mu\in C_{n,k_3}\cap
D_{n,k_4}$ for some pairwise different $1\leq k_i\leq 4$. In
particular $\la,\mu\in\cup_{k=1}^4 C_{n,k}$. In each of these cases
it can be checked that
$\chi^\la(\tau_{(n-5,5)})=\chi^\mu(\tau_{(n-5,5)})=0$.

Next assume that $\la\in\{(n-3,1^3),(4,1^{n-4})\}$. Then $\mu\in
N_{n,1}\cap N_{n,2}\cap N_{n,3}=\{(n),(1^n)\}$, leading to a
contradiction.

If $\la\in\{(n-2,1^2),(3,1^{n-3})\}$, then
\[\mu\in N_{n,1}\cap N_{n,2}=\{(n),(n-2,2),(2^2,1^{n-4}),(1^n)\},\]
so that $\mu\in\{(n-2,2),(2^2,1^{n-4})\}$, which had been excluded
at the beginning of the proof.

If $\la\in\{(n-1,1),(2,1^{n-2})\}$ then $\mu\in
N_{n,1}\setminus\{(n),(1^n)\}$ and so $\mu=(n-h,2,1^{h-2})$ for some
$2\leq h\leq n-2$. In this case
$\chi^\la(\tau_\ga)=\chi^\mu(\tau_\ga)=0$ for some
$\ga\in\{(n-5,2^2,1),(n-6,3,2,1)\}$.

\medskip

{\bf Case 2:} $n$ is odd. In this case
$\tau_{(n)},\tau_{(n-k,k-1,1)}\in\AAA_n$ for $2\leq k\leq(n+1)/2$.
Let $N_{n,0}:=\{\ga\vdash n|\chi^\ga(\tau_{(n)})=0\}$ and for $2\leq
k\leq 4$ let $N_{n,k}:=\{\ga\vdash
n|\chi^\ga(\tau_{(n-k,k-1,1)})=0\}$. We can write
$N_{n,k}=B_{n,k}\cup C_{n,k}\cup D_{n,k}$ with
\begin{align*}
B_{n,0}=&\emptyset,\\
C_{n,0}=&\emptyset,\\
D_{n,0}=&\{(n-h,1^h)|0\leq h\leq n-1\},\\
B_{n,2}=&\{(n),(n-2,2),(2^2,1^{n-4}),(1^n)\},\\
C_{n,2}=&\{(n-1,1),(2,1^{n-2})\},\\
D_{n,2}=&\{(n-h,3,1^{h-3})|3\leq h\leq n-3\}\cup\{(n-h,2^2,1^{h-4})|4\leq h\leq n-2\},\\
B_{n,3}=&\{(n),(n-3,2,1),(3,2,1^{n-5}),(1^n)\},\\
C_{n,3}=&\{(n-2,1^2),(n-4,2^2),(3^2,1^{n-6}),(3,1^{n-3})\},\\
D_{n,3}=&\{(n-h,4,1^{h-4})|4\leq h\leq n-4\}\cup\{(n-h,2^3,1^{h-6})|6\leq h\leq n-2\},\\
B_{n,4}=&\{(n),(n-2,2),(n-3,3),(2^3,1^{n-6}),(2^2,1^{n-4}),(1^n)\},\\
C_{n,4}=&\{(n-3,1^3),(n-4,2,1^2),(n-5,2^2,1),(n-6,2^3),(4^2,1^{n-8}),(4,3,1^{n-7}),\\
&(4,2,1^{n-6}),(4,1^{n-4})\},\\
D_{n,4}=&\{(n-h,5,1^{h-5})|5\leq h\leq n-5\}\cup\{(n-h,3^2,1^{h-6})|6\leq h\leq n-3\}\\
&\cup\{(n-h,2^4,1^{h-8})|8\leq h\leq n-2\}.
\end{align*}
This can be seen again by noting that if $\ga\in N_{n,k}$ then $\ga$
can be obtained by adding a $(n-k)$-hook to a partition
$\overline{\ga}\vdash k$. Further if $2\leq k\leq 4$ then
$\chi^{\overline{\ga}}(\tau_{(k-1,1)})\not=0$. This conditions are
again sufficient for $\ga\in N_{n,k}$ since $k<n/2$.

Again the sets $D_{n,k}$ are pairwise disjoint, the same holds for
the sets $C_{n,k}$ since $n\geq 11$ and we may assume that
$\{\la,\mu\}\cap N_{n,k}\not=\emptyset$ for $k\in\{0,2,3,4\}$.
Assume first that neither $\la$ nor $\mu$ is contained in $\cup
B_{n,k}$. Then similarly to Case 1, $\la,\mu\in\cup C_{n,k}$. In
particular $\chi(\tau_{(n-5,4,1)})=0$ and
$\psi(\tau_{(n-5,4,1)})=0$. So we may now assume that $\la\in(\cup
B_{n,k})\setminus \{(n),(1^n)\}$.

If $\la\in\{(n-2,2),(2^2,1^{n-2})\}$ then $\mu\in N_{n,0}\cap N_{n,3}$
and then $\mu\in\{(n-2,1^2),(3,1^{n-3})\}$, which had been excluded
at the beginning of the proof.

If $\la\in\{(n-3,3),(2^3,1^{n-6})\}$ then again $\mu\in N_{n,0}\cap N_{n,3}$, so $\mu\in\{(n-2,1^2),(3,1^{n-3})\}$ and then
$\chi^\la(\tau_{(n-5,4,1)})=\chi^\mu(\tau_{(n-5,4,1)})=0$.

If $\la\in\{(n-3,2,1),(3,2,1^{n-5})\}$ then $\mu\in N_{n,0}\cap
N_{n,2}\cap N_{n,4}=\{(n),(1^n)\}$, leading to a contradiction.
\end{proof}

We will now prove Theorem \ref{main:A}:

\begin{proof}[Proof of Theorem A]
We will check that $\chi(1)=n(n-3)/2$ if and only if $\chi\in\{\chi^{(n-2,2)},\chi^{(2^2,1^{n-4})}\}$ and that $\chi(1)=(n-1)(n-2)/2$ if and only if $\chi\in\{\chi^{(n-2,1^2)},\chi^{(3,1^{n-3})}\}$. The theorem will then follow by Theorem \ref{L160123}.

The `if' parts are easily checked using the hook formula. The `only if' parts for $n=8$ are also easily checked. So assume now that $n\geq 9$. Note that $\chi^{(n)}(1),\chi^{(1^n)}(1)=1$ and $\chi^{(n-1,1)}(1),\chi^{(2,1^n)}(1)=n-1$. We will show that $\chi^\la(1)>(n-1)(n-2)/2$ for any $\la\vdash n$ with $\la_1,\la_1'\leq n-3$ (with $\la'$ the partition which is conjugated to $\la$), which will conclude the proof of the `only if' parts.

For $n=9$ and $10$ this can easily be checked looking at character
tables. So assume that $n\geq 11$ and that the claim holds for $n-1$
and $n-2$. Then by induction $\chi^\mu(1)\geq (n-1)(n-4)/2$ for any
$\mu\vdash n-1$ with $\mu_1,\mu_1'\leq n-3$ and
$\chi^\nu(1)>(n-3)(n-4)/2$ for any $\nu\vdash n-2$ with
$\nu_1,\nu_1'\leq n-5$. If $\la$ has at least 2 removable nodes $A$
and $B$ then, by the branching rule,
\[\chi^\la(1)\geq\chi^{\la\setminus\{A\}}(1)+\chi^{\la\setminus\{B\}}(1)\geq (n-1)(n-4)>(n-1)(n-2)/2.\]
If $\la$ has only one removable node then $\la=(a^b)$ for some $a,b$ with $ab=n$. Further $2\leq a,b\leq n/2<n-5$. So
\[\chi^\la(1)=
\chi^{(a^{b-2},(a-1)^2)}(1)+\chi^{(a^{b-1},a-2)}(1)>(n-3)(n-4)>(n-1)(n-2)/2,\]
which concludes the proof.
\end{proof}

We will next prove the following result, which implies Theorem \ref{main:D} for alternating groups (if $n=5$ or $6$ it can be checked that $\Gamma_v(\AAA_n)$ has $3$ resp. 2 connected components by looking at character tables):

\begin{thm}\label{T280223}
Let $n\geq 7$ then each of $\Gamma_v(\SSS_n)$, $\Gamma_v(\AAA_n)$, $\Delta_v(\SSS_n)$ and $\Delta_v(\AAA_n)$ is connected. Further $\Gamma_v(\SSS_n)$ and $\Gamma_v(\AAA_n)$ have diameter 2 and $\Delta_v(\SSS_n)$ and $\Delta_v(\AAA_n)$ diameter at most 2.
\end{thm}

Note that by the above theorem (and checking small cases), we have that $\Gamma(\SSS_n)\subseteq\Gamma_v(\SSS_n)$ and $\Gamma(\AAA_n)\subseteq\Gamma_v(\AAA_n)$. In order to prove the theorem we need the following lemma:

\begin{lem}\label{L010223}
Let $n\geq 7$ and $\si\in\SSS_n$ be a vanishing element.
Then there exists
$\chi\in\Irr(\SSS_n)\setminus\{\chi^{(n-2,2)},\chi^{(2^2,1^{n-4})}\}$
with $\chi(\si)\not=0$.
\end{lem}

\begin{proof}
For $n\leq 9$ the result can be proved looking at character tables.
So assume from now on that $n\geq 10$.

Let $\tau$ be the cycle partition of $\si$. If $\tau_1\geq 4$ then
there exists a $\tau_1$-core $\la$ of $n$ by \cite[Theorem 1]{GO}.
In particular $\chi^\la(\si)=0$ in view of the Murnaghan-Nakayama
formula. If $\tau_1\leq n-4$ then neither $(n-2,2)$ nor
$(2^2,1^{n-4})$ is a $\tau_1$-core. If $\tau_1\geq n-3$ then we can
take $\la=(n-5,5)$ as $\tau_1$-core.

So we may now assume that $\tau=(3^a,2^b,1^c)$ for some $a,b,c\geq
0$ with $3a+2b+c=n$. Taking $\la=(n-4,2,1^2)$, $(n-1,1)$ or
$(n-3,1^2)$ depending on whether $n$ is congruent to 0, 1 or 2
modulo 3 respectively, we may assume that $n-3a\geq 4$ (since for this choice of
$\la$ we have that $|\la_{(3)}|\geq 4$, so that $\chi^\la(\si)=0$ if
$n-3a\leq 3$).

If $b$ is odd then $\chi^\la(\tau)=0$ for any $\la=\la'$, in particular for $\la=(n/2,2,1^{n/2-2})$ or $((n+1)/2,1^{(n-1)/2})$ depending on the parity of $n$, since then $\chi^\la=\chi^\la\cdot\sgn$.
So we may assume that $b$ is even.

Using the determinantal formula we have that
\[\chi^{(n-2,2)}=1\ua_{\SSS_{n-2,2}}^{\SSS_n}-1\ua_{\SSS_{n-1}}^{\SSS_n},\]
where $\SSS_{n-2,2}\cong\SSS_{n-2}\times\SSS_2$ is a maximal Young subgroup.
We will now show that $\chi^{(n-2,2)}(\si)\not=0$. Multiplying with
$\sgn$ this also gives $\chi^{(2^2,1^{n-4})}(\si)\not=0$, so that
the lemma follows.

By the above formulas we have that
\[\chi^{(n-2,2)}(\si)=b+\binom{c}{2}-c=b+\frac{c(c-3)}{2}\]
So $\chi^{(n-2,2)}(\si)\not=0$ unless $(b,c)\in\{(0,0),(0,3),(1,1),(1,2)\}$. Each of these choices contradicts $2b+c=n-3a\geq 4$ or $b$ even.
\end{proof}

\begin{proof}[Proof of Theorem \ref{T280223}]
For $n=7$ or $8$ the theorem can be easily checked. So assume that $n\geq 9$. By  Theorem \ref{L160123} we have that $\Gamma_v(\SSS_n)$ is connected with diameter 2. Further if $g,h\in\SSS_n$ are vanishing elements, then by Lemma \ref{L010223} there are $\chi,\psi\in\Irr(\SSS_n)\setminus\{\chi^{(n-2,2)},\chi^{(2^2,1^{n-4})}\}$ with $\chi(g)=0$ and $\psi(h)=0$. Since $\chi$ and $\psi$ have a common zero, it follows that $\Delta_v(\SSS_n)$ is connected and that it has diameter at most 2.

Let now $\chi,\psi\in\Irr(\AAA_n)$ have degree larger than 1. Then $\chi$ and $\psi$ have a common zero by Lemma \ref{L130123} and Theorem \ref{L160123} if and only if $\{\chi,\psi\}\not=\{(\chi^{(n-2,2)})_{\AAA_n},(\chi^{(n-2,1^2)})_{\AAA_n}\}$ (note that $(\chi^{(n-2,2)})_{\AAA_n}$ and $(\chi^{(n-2,1^2)})_{\AAA_n}$ are both irreducible by \cite[Theorem 2.5.7]{jk}). Further any vanishing element $g\in\AAA_n$ is a zero of some irreducible character $\not=(\chi^{(n-2,2)})_{\AAA_n}$ by Lemmas \ref{L130123} and \ref{L010223}. We can then conclude as in the $\SSS_n$ case.
\end{proof}

\section{Solvable groups}\label{sec:solvable}


Here we prove Theorem \ref{main:C}. We begin with the easy case of
nilpotent groups.

\begin{thm}
\label{pgp} Let $G$ be a nilpotent group. Then any two non-linear
irreducible characters of $G$ share a common zero. Equivalently,
$\Gamma_v(G)$ is complete.
\end{thm}

\begin{proof}
Let $\chi,\psi\in\Irr(G)$ be non-linear. Since nilpotent groups are
monomial and maximal subgroups of nilpotent groups are normal, we
have that there exist normal maximal  subgroups $M$ and $N$ of $G$
such that $\chi$ is induced from $M$ and $\psi$ is induced from $N$.
By the character-induction formula, both characters vanish outside
$M\cup N$. Since $G\neq M\cup N$ by comparing orders, the result
follows.
\end{proof}

Theorem \ref{pgp} fails when the group is not nilpotent. The group
$\SL_2(3)$ is already a counterexample. In fact,
$\Gamma_v(\SL_2(3))$ is disconnected. We also note that the graph
$\Delta_v(G)$ (defined in Section \ref{sec:classes}) for $G$
nilpotent does not need to be complete, as the dihedral group
$\DDD_{16}$ shows.

Next, we consider nilpotent-by-abelian groups. As usual, $\bF(G)$
denotes the Fitting subgroup of $G$.

\begin{lem}
\label{ro} Let $G$ be a finite group such that $G/\bF(G)$ is
abelian. Then there exists $\lambda\in\Irr(\bF(G))$ such that
$\chi=\lambda^G\in\Irr(G)$. In particular, $\chi$ vanishes on
$G-\bF(G)$ and all the $G$-classes in $G-\bF(G)$ are linked in
$\Delta_v(G)$.
\end{lem}

\begin{proof}
This follows from the proof of Lemma 18.1 of \cite{mw}.
\end{proof}

With a slight abuse of language, sometimes we say that $x,y\in G$
are linked in $\Delta_v(G)$ to mean that the conjugacy classes of
$x$ and $y$ are linked in $\Delta_v(G)$, or that $x$ belongs to a
connected component of $\Delta_v(G)$ to mean that the class of
$x$ belongs to that connected component.

\begin{lem}
\label{abel} Let $G$ be a finite group such that $G/\bF(G)$ is
abelian. Then $\Delta_v(G)$ has at most two connected components,
one of which contains all the classes in $G-\bF(G)$. If there are
two connected components, then the second one contains all the
vanishing classes in $\bF(G)-\bZ(\bF(G))$.
\end{lem}

\begin{proof}
By Lemma \ref{ro}, all the classes in $G-\bF(G)$ are linked. Let
$\Delta_1$ be the connected component containing these classes.
Suppose first that $\bF(G)$ is abelian. Then Lemma~\ref{res} implies
that, for any $\chi\in\Irr(G)$ non-linear,
$\Van(\chi)\not\subseteq\bF(G)$. Therefore, all the vanishing
classes are linked to some class in $G-\bF(G)$, which implies that
$\Delta_v(G)$ is connected.

Suppose now that $\bF(G)$ is not abelian and that $\Delta_v(G)$ is
not connected. Let $\Delta_1,\dots,\Delta_t$ be the connected
components of the graph, where $\Delta_1$ contains all classes in
$G-\bF(G)$. So all the classes in $\Delta_2,\dots,\Delta_t$ are
contained in $\bF(G)$. Let $y$ be a representative of a class in one
of these components. Therefore, $y$ is a vanishing element and for
every $\chi\in\Irr(G)$ satisfying $\chi(y)=0$,
\[\chi_{\bF(G)}\in\Irr(\bF(G)),\] by Lemma \ref{res}. In
particular, $y$ is a vanishing element of $\bF(G)$, i.e.,  all the
classes in $\Delta_2,\dots,\Delta_t$ are contained in
$\bF(G)-\bZ(\bF(G))$. Suppose that $y_1,y_2\in \bF(G)-\bZ(\bF(G))$
lie in vanishing classes. Let $\chi_i\in\Irr(G)$ such that
$\chi_i(y_i)=0$. Therefore,
\[\varphi_i:=(\chi_i)_{\bF(G)}\in\Irr(\bF(G)).\] By Theorem \ref{pgp},
there exists $t\in\bF(G)$ such that $\varphi_i(t)=0$ for $i=1,2$.
But then $y_1$ and $t$ are linked by means of $\chi_1$ and $t$ and
$y_2$ are linked by means of $\chi_2$. Thus $y_1$  and $y_2$ belong
to the same connected component.  It follows that all the vanishing
classes in $\bF(G)$ belong to the same connected component. The
result follows.
\end{proof}

Recall that the Fitting series of a finite group $G$ is the sequence
of characteristic subgroups $\bF_i(G)$ defined by $\bF_0(G)=1$, $\bF_1(G)=\bF(G)$  and
$\bF_{i+1}(G)/\bF_i(G)=\bF(G/\bF_i(G))$ for $i\geq1$. If $G$ is solvable, there exists an integer  $n$ such that
$\bF_n(G)=G$. The smallest such integer  is called the Fitting height of $G$.
As usual, given $N\trianglelefteq G$ and $\theta\in\Irr(N)$ we write
$\Irr(G|\theta)$ to denote the set of irreducible characters of $G$ lying over
$\theta$.

\begin{lem}
\label{meta} Let $G$ be a solvable group of Fitting height $n$. If
$G/\bF_{n-1}(G)$ is not abelian, then $\Delta_v(G)$   has at most
two connected components.
\end{lem}

\begin{proof}
We may assume that $n>1$. Set $F:=\bF_{n-1}(G)$ and let
$Z/\bF_{n-1}(G)=\bZ(G/\bF_{n-1}(G))$. By Lemma \ref{vn} and Theorem
\ref{pgp}, we know that all the classes in $G-Z$ belong to the same
connected component $\Delta_1$. Since $Z/F$ is abelian,
Lemma~\ref{ro} applied to $Z/\bF_{n-2}(G)$ implies that there exists
$\varphi\in\Irr(Z)$ that vanishes on  all the  classes  in $Z-F$.
Therefore, the same holds for any $\chi\in\Irr(G|\varphi)$.  In
particular, all the classes in $Z-F$ belong to the same connected
component $\Delta_2$ (possibly $\Delta_1=\Delta_2$). We claim that
any vanishing element $x\in F$ belongs to either $\Delta_1$ or
$\Delta_2$. Let $\chi\in\Irr(G)$ such that $\chi(x)=0$. We may
assume that $\Van(\chi)\subseteq F$. By Lemma~\ref{res},
$\chi_F\in\Irr(F)$. Let $\psi\in\Irr(G/F)$ non-linear. By
Gallagher's theorem \cite[Corollary 6.17]{isa}, $\chi\psi\in\Irr(G)$
vanishes at both $x$ and some element in $G-F$. This proves the
claim. Thus $\Delta_1$ and $\Delta_2$ are all the connected
components.
\end{proof}

We can now complete the proof of Theorem \ref{main:C}.

\begin{thm}
\label{sol} Let $G$ be a solvable group. Then $\Gamma_v(G)$ has at
most two connected components.
\end{thm}

\begin{proof}
By Lemma \ref{rel} it suffices to show that $\Delta_v(G)$ has at
most two connected components. Let $n$ be the Fitting height of $G$.
By Theorem \ref{pgp}, we may assume that $n>1$. By Lemma \ref{meta},
we may assume that $G/\bF_{n-1}(G)$ is abelian. By Lemma \ref{ro},
all the classes in $G-\bF_{n-1}(G)$ belong to the same connected
component, say $\Delta_1$. By Lemma \ref{abel}, we may assume that
$n\geq 3$.

Suppose first that $\bF_{n-1}(G)/\bF_{n-2}(G)$ is abelian. Using
Lemma \ref{ro} again, the classes in $\bF_{n-1}(G)-\bF_{n-2}(G)$
belong to the same connected component, say $\Delta_2$ (possibly
$\Delta_2=\Delta_1$). Suppose that $x\in\bF_{n-2}(G)$ is a vanishing
element. We want to see that the class of $x$ belongs to $\Delta_1$
or $\Delta_2$. Suppose not. Then, if $\chi\in\Irr(G)$ is such that
$\chi(x)=0$, we have that $\chi$ does not have any zeros in
$G-\bF_{n-2}(G)$. By Lemma~\ref{res}, it follows that
$\chi_{\bF_{n-2}(G)}\in\Irr(\bF_{n-2}(G))$. Using Gallagher's
theorem, we deduce that $\chi\psi\in\Irr(G)$ for every
$\psi\in\Irr(G/\bF_{n-2}(G))$. We can assume that $\psi$ is
non-linear. Since this character has zeros on $G-\bF_{n-2}(G)$, we
have a contradiction.

Finally, we may assume that  $\bF_{n-1}(G)/\bF_{n-2}(G)$ is not
abelian. In this case, we will prove that $\Gamma_v(G)$ has at most
two connected components. Let $\Phi$ be the (normal) subgroup of $G$
such that \[\Phi/\bF_{n-2}(G)=\Phi(G/\bF_{n-2}(G)).\] By Lemma
\ref{ro} applied to $G/\Phi$, there exists
$\lambda\in\Irr(\bF_{n-1}(G)/\Phi)$ such that
$$\psi:=\lambda^G\in\Irr(G).$$ Clearly, by Lemma \ref{res}, all the
irreducible characters of $G$ whose restriction to $\bF_{n-1}(G)$ is
not irreducible belong to the connected component of $\psi$. Now we
claim that all the characters in $$\mathcal{A}=
\{\chi\in\Irr(G)\mid\chi_{\Phi}\in\Irr(\Phi),\chi(1)>1\}$$ belong
to the  connected component of $\psi$  and that all the characters
in
$$\mathcal{B}=\{\chi\in\Irr(G)\mid\chi_{\bF_{n-1}}(G)\in\Irr(\bF_{n-1}(G)),
\chi_{\Phi}\not\in\Irr(\Phi)\}$$  belong to the same connected
component. The result will follow.

Let $\chi\in\mathcal{A}$. By the definition of $\mathcal{A}$,
$\chi_{\Phi}\in\Irr(\Phi)$. By Gallagher again,
$\chi\psi\in\Irr(G)$, which implies that $\chi$ and $\chi\psi$ are
linked in $\Gamma_v(G)$. Since $\chi\psi$ vanishes on
$G-\bF_{n-1}(G)$, we deduce that $\chi\psi$ is linked to $\psi$.
Therefore, there is a path of length $2$ joining $\chi$ and $\psi$.
We have thus seen that the characters in $\mathcal{A}$ and $\psi$
are in the same connected component.

It remains to prove that the characters in $\mathcal{B}$ belong to
the same connected component. First, we see that for any
$\chi\in\mathcal{B}$, there exists $U_{\chi}\trianglelefteq G$ such
that $\Phi\leq U_{\chi}<\bF_{n-1}(G)$ and $\chi$ vanishes on
$\bF_{n-1}(G)-U_{\chi}$. Let $\varphi\in\Irr(\Phi)$ lying under
$\chi$. Suppose that $\varphi$ is $\bF_{n-1}(G)$-invariant, so that
$(\bF_{n-1}(G),\Phi,\varphi)$ is a character triple with
$\bF_{n-1}(G)/\Phi=\bF(G/\Phi)$ abelian by Gasch\"utz's theorem \cite[Theorem 1.12]{mw}.
The existence of $U_{\chi}$ follows from Lemma 2.2 of \cite{wol}. Now, we assume that $\varphi$ is not
$\bF_{n-1}(G)$-invariant. Let $T:=I_G(\varphi)$ and note that
$G=T\bF_{n-1}(G)$ (because $\chi$ restricts irreducibly to
$\bF_{n-1}(G)$).  In particular, $T\cap \bF_{n-1}(G)\trianglelefteq
G$ (because $T\cap\bF_{n-1}(G)\trianglelefteq T$,
$\bF_{n-1}(G)/\Phi$ is abelian and $T$ contains $\Phi$, so $T\cap \bF_{n-1}(G)$
 is also normal in  $\bF_{n-1}(G)$).
 Set $U_{\chi}=T\cap  \bF_{n-1}(G)$ and note that it satisfies
the properties that we want.

Therefore, if $\chi,\psi\in\mathcal{B}$, then both characters vanish
on $\bF_{n-1}(G)-(U_{\chi}\cup U_{\psi})$. Since $$U_{\chi}\cup
U_{\psi}\subsetneq\bF_{n-1}(G),$$ we conclude that all the characters
in $\mathcal{B}$ are linked. This completes the proof.
\end{proof}

As mentioned at the end of Section \ref{sec:classes}, there are
examples of solvable groups with $\Delta_v(G)$ disconnected. The
Fitting height of the group appears to be a relevant factor to
decide the connectedness of $\Delta_v(G)$. The solvable examples
mentioned there have Fitting height $3$. There are also similar
examples of Fitting height $3$ among odd order groups. Let
$H\leq\GL_3(3)$ be a Frobenius group of order $39$ and let $G=HV$ be
the semidirect product of $H$ acting on $V$, where $V$ is the
natural module for $\GL_3(3)$. It is easy to check that this group
has Fitting height $3$ and $\Delta_v(G)$ is disconnected. We
conclude this section with the following.

\begin{thm}
\label{large}
Let $G$ be a solvable group. Suppose that either the Fitting height
of $G$ exceeds $9$ or that $G$ has odd order and Fitting height at
least $5$. Then $\Delta_v(G)$ is connected and has diameter at most
$2$.
\end{thm}

\begin{proof}
By \cite[Theorem 5.2]{yan}, there exists $\mu\in\Irr(\bF_8(G))$ such
that $\chi=\mu^G\in\Irr(G)$. In particular, $\chi$ vanishes on
$G-\bF_8(G)$. Thus all classes in $G-\bF_8(G)$ belong to the same
connected component $\Delta_1$.  Furthermore, they are linked. Let
$x\in\Van(G)$ and assume that  $x$ is not linked to any class in
$G-\bF_8(G)$. Let $\psi\in\Irr(G)$ such that $\psi(x)=0$. Therefore,
$\Van(\psi)\subseteq\bF_8(G)$. Lemma \ref{res} implies that
$\psi_{\bF_8(G)}\in\Irr(\bF_8(G))$. Thus $\psi\gamma\in\Irr(G)$ for
every $\gamma\in\Irr(G/\bF_8(G))$ non-linear. This character
vanishes both at $x$ and at some element in $G-\bF_8(G)$. This is a
contradiction.

Suppose now that $|G|$ is odd. By \cite[Theorem D]{mowo}, there
exists $\chi\in\Irr(G)$ such that $\chi$ vanishes on $G-\bF_3(G)$.
The result follows by similar arguments as above.
\end{proof}

We conjecture that if $G$ is solvable and $\Delta_v(G)$ is disconnected
then the Fitting height of $G$ is at most $3$.


\section{Groups of Lie type}\label{sec:Lie type}

In this section we complete the proof of Theorem \ref{main:D}.

In \cite[Theorem 5.1]{mno}, Malle, Navarro, and Olsson proved that,
for every finite simple group $G$ of Lie type, there exist four
conjugacy classes (of elements of prime order) in $G$ such that
every non-trivial irreducible character of $G$ vanishes on at least
one of them. This instantly shows that the common-zero graph
$\Gamma_v(G)$ of $G$ has at most four connected components. With
some more work, this bound can be lowered to 3. Note that 3 is the
best possible bound, as shown by $\PSL_2(q)$ for several choices of
$q$.

In fact, it is known that, with the possible exception of
$G=\mathrm{P}\Omega_{2n}^+(q)$, every simple group of Lie type has a
pair of conjugacy classes, say $(C,D)$, called \emph{strongly
orthogonal} pair, such that
\[\chi(C)\chi(D)=0\] for every $\chi\in\Irr(G)$ but only two
characters. One of them, of course, is the trivial character
$\mathbf{1}_G$ and the other is usually the Steinberg one $\St_G$.
This was done in the proofs of Theorems 2.1-2.6 of \cite{mallesw}
for classical groups and in \cite[\S 10]{lm} for groups of
exceptional types. In this case, $\Gamma_v(G)$ clearly has at most
three connected components.

We are left with only one family $G=\mathrm{P}\Omega_{2n}^+(q)$. It
is worth noting that the common-zero graph of several simple groups
of Lie type is indeed connected, although we have not made an effort
to make this precise in previous cases. We take the opportunity of
this remaining case to prove the connectedness of the graph.

Recall that, for $p$ a prime, a $p$-defect zero character of $G$ is
a character with degree divisible by $|G|_p$. We will use a
well-known fact that $p$-defect zero irreducible characters vanish
on every $p$-singular element. It follows that if $p$ divides $|G|$
then all the $p$-defect zero characters of $G$ share a common zero.
To see that a character is of $p$-defect zero, we frequently use a
case of Zsigmondy's theorem stating that, for every $n\in\ZZ^{\geq
2}$ and $q\in \ZZ^{\geq 2}$ with $(n,q)\neq (6,2)$ and $q+1$ is not
a $2$-power when $n=2$, there is a prime (called primitive prime
divisor) that divides $q^n-1$ and does not divide $q^k-1$ for any
positive integer $k<n$. Following \cite{mno}, we denote such a prime
by $\ell(n)$.

Orders and character degrees of $G$ are conveniently expressed as
products of a power of $q$ and cyclotomic polynomials $\Phi_i$
evaluated at $q$, up to a constant. Note that $\ell(n)$ can be
defined as a prime dividing $\Phi_n$ but not $\Phi_k$ for any $k<n$.

\begin{thm}
\label{lie} Let $G$ be a finite simple group of Lie type. Then
$\Gamma_v(G)$ has at most three connected components.
\end{thm}

\begin{proof}
As mentioned above, we may assume that
$G=\mathrm{P}\Omega_{2n}^+(q)$ with $n\geq 4$, and we aim to prove
that $\Gamma_v(G)$ is connected. Let $G_{sc}$ be the corresponding
finite reductive group of simply connected type, so that
$G_{sc}=\mathrm{Spin}_{2n}(q)$ is the full covering group of $G$ and
$G=G_{sc}/\bZ(G_{sc})$.

Maximal tori and their orders of finite reductive groups are
well-known, see e.g. \cite[\S 3A]{mal}. Here a maximal torus of
$G_{sc}$ is defined to be the $F$-fixed points of a maximal torus of
the ambient algebraic group $\bG$ under a suitable Frobenius map
$F:\bG\rightarrow \bG$ such that $G_{sc}=\bG^F$. Specifically, the
$G_{sc}$-conjugacy classes of $F$-stable maximal tori of $\bG$ are
parameterized by pairs of partitions $(\lambda,\mu)$ of $n$ (that
is, $\lambda=(\lambda_1,\lambda_2,...)$ and $\mu=(\mu_1,\mu_2,...)$
with $\sum_i \lambda_i+\sum_j \mu_j=n$) such that the number of
parts of $\mu$ is even. The order the corresponding (conjugate)
maximal tori of $G_{sc}$ is
\[
\prod_{\lambda_i}(q^{\lambda_i}-1)\prod_{\mu_j}(q^{\mu_j}+1).
\]

Consider three tori $\mathcal{T}_i$ ($1\leq i\leq 3$) of $G_{sc}$ of
orders
\[
|\mathcal{T}_1|=(q^{n-1}+1)(q+1),
|\mathcal{T}_2|=(q^{n-2}+1)(q^2+1),\] and
\begin{equation*}
|\mathcal{T}_3| = \begin{cases}
q^{n}-1 &\text{ if } n \text{ is odd}\\
(q^{n-1}-1)(q-1) &\text{ if } n \text{ is even}.
\end{cases}
\end{equation*}

Assume for a moment that $n\geq 5$ and $(n,q)\neq(5,2)$. In
particular, primitive prime divisors \[\ell_1:=\ell(2n-2),
\ell_2:=\ell(2n-4),\] and
\begin{equation*}
\ell_3:= \begin{cases}
\ell(n) &\text{ if } n \text{ is odd}\\
\ell(n-1) &\text{ if } n \text{ is even}.
\end{cases}
\end{equation*}
exist. Furthermore, $\ell_i$ divides $|\mathcal{T}_i|$ and the
$\ell_i$-Sylow subgroups of $G_{sc}$ are cyclic (see \cite[Theorem
25.14]{Malle-Testerman-book}). Let $g_i\in \mathcal{T}_i$ be of
order $\ell_i$. We will use the same notation $g_i$ and
$\mathcal{T}_i$ for their images under the natural projection from
$G_{sc}$ to $G$.

We first argue that all the non-unipotent characters of $G$ (as well
as of $G_{sc}$) are contained in one connected component. (Here, a
character of $G$ is called non-unipotent if its lift to $G_{sc}$ is
non-unipotent. See \cite[Definition 13.19]{DM} for the definition of
unipotent characters of finite reductive groups.) For this it is
sufficient to show that any such character of $G_{sc}$ vanishes on
at least two of $g_i$. This can be argued similarly as in \cite[\S
3B]{mal}. Assume otherwise. Then there is $\chi$ belonging to the
Lusztig series $\EC(G_{sc},s)$ for some nontrivial semisimple
element \[s\in G_{ad}:=\mathrm{P}(\mathrm{CO}_{2n}(q)^0)\] such that
$\chi$ is non-zero on at least two of $g_i$.

Suppose that $\chi(g_1)\neq 0$. Then $\ell_1$ does not divide
$\chi(1)$ (otherwise, since every character degree is a product of a
power of $q$ and some of $\Phi_k$ up to a constant, we have that
$\Phi_{2n-2}$ appears in the product for $\chi(1)$, which would
imply that $\chi$ is of $\ell_1$-defect zero, and hence vanishes at
$g_1$, a contradiction). It follows from the character-degree
formula in Lusztig's parametrization \cite[Remark 13.24]{DM} that
$|\bC_{G_{ad}}(s)|$ is divisible by $\ell_1$ and moreover
$\bC_{G_{ad}}(s)$ contains a conjugate of $\mathcal{T}_1^\ast$,
where $\mathcal{T}_i^\ast$ is the torus of $G_{ad}$ dual to
$\mathcal{T}_i$. What we have shown also applies to the cases
$\chi(g_2)\neq 0$ and $\chi(g_3)\neq 0$. Since $\chi$ is non-zero on
at least two of $g_i$, we deduce that $\bC_{G_{ad}}(s)$ contains
certain conjugates of at least two of $\mathcal{T}_i^\ast$. Using
the known structure of centralizers of semisimple elements in finite
reductive groups (see \cite[Lemmas 2.3 and 2.5]{Nguyen10}, for
instance, for the case of split orthogonal groups), we see that $s$
must be trivial, violating the non-unipotent assumption on $\chi$.

Lusztig's classification of ordinary irreducible characters of
finite reductive groups, together with the aforementioned
character-degree formula and the known centralizers of semisimple
elements also show that, for each $i$, $G$ possesses a non-unipotent
character of $\ell_i$-defect $0$, and thus vanishes on $g_i$. For
instance, for $i=1$, we choose a semisimple element $s\in G
=[G_{ad},G_{ad}]$ so that $\Phi_{2n-2}$, a polynomial in $q$, is not
a factor of $|\mathbf{C}_{G_{ad}}(s)|$. Every character in the
Lusztig series $\EC(G_{sc},s)$ then have degree divisible by
$\Phi_{2n-2}$ and $\ell_1$-defect $0$. Moreover, these characters
(of $G_{sc}$) restrict trivially to $\bZ(G_{sc})$ (see \cite[Lemma
5.8]{Hu}, for instance), and therefore they are lifts of characters
of $G$.

We now turn to (nontrivial) unipotent characters. First assume that
$n$ is odd. As mentioned in \cite[\S 3G]{mal}, all these characters
except the Steinberg one have degree divisible by either $\ell_1$ or
$\ell_3$, and therefore have either $\ell_1$-defect or
$\ell_3$-defect zero, and thus vanish on either $g_1$ or $g_3$. We
now know that all members of $\Irr(G)\backslash
\{\mathbf{1}_G,\St_G\}$ are contained in just one connected
component of $\Gamma_v(G)$. Thus we would be done if $\St_G$ has a
common zero with any other irreducible character of the group. This
is not difficult to see. Let $r$ be the defining characteristic of
the group. Consider $g\in G$ that is an $r$-singular element but not
an $r$-element and $p\neq r$ is a prime divisor of $|g|$. Then $g$
is a vanishing element for both $\St_G$ and any $p$-defect zero
characters.

Now assume that $n$ is even. According to \cite[\S 3G]{mal}, if a
nontrivial unipotent character of $G$ has degree not divisible by
either $\ell_1$ or $\ell_3$, it must be either the Steinberg
character or one of the two others labeled by the symbols
$${n-1 \choose 1} \text{ and } {0\hspace{6pt}\cdots \hspace{6pt} n-3\hspace{6pt}n-1
\choose 1\hspace{6pt}\cdots\hspace{6pt}n-2\hspace{6pt} n-1}.$$ (We
refer the reader to \cite[\S 13.8]{Carter} for the labeling and
degree formulas of unipotent characters of classical groups.) The
degrees of these two characters, however, are divisible by
$\Phi_{2n-4}$. They are therefore of $\ell_2$-defect zero, and thus
vanish at $g_2$, proving that they are in the same connected
component with non-unipotent characters. As with the case of odd
$n$, $\Gamma_v(G)$ is therefore connected.

Consider $G=\mathrm{P}\Omega^+_8(q)$. As the case $(n,q)=(4,2)$ can
be checked using \cite{gap}, we assume that $q>2$, so that the
primitive prime divisors $\ell_1,\ell_2,\ell_3$ still exist. (Note
that $|G|$ is now divisible by $\Phi_4^2$ and the condition
$\chi(g_2)\neq 0$ does not imply that $\bC_{G_{ad}}(s)$ contains a
conjugate of $\mathcal{T}_2^\ast$, as we had earlier.) We still have
that every non-unipotent character of $G$ vanishes at either $g_1$
or $g_3$. On the other hand, the (two) unipotent characters of $G$
labeled by the symbol ${2 \choose 2}$ have degree $q^2\Phi_3\Phi_6$,
and hence are of both $\ell_1$- and $\ell_3$-defect zero. They
therefore vanish on both $g_1$ and $g_3$, and it follows that all
the non-unipotent characters of $G$ are contained in one connected
component of $\Gamma_v(G)$. As mentioned above, this connected
component also contains all (nontrivial) unipotent characters except
possibly the Steinberg one or the two characters labeled by ${3
\choose 1}$  and ${0\hspace{6pt}1 \hspace{6pt} 3 \choose
1\hspace{6pt}2\hspace{6pt}3}$, which are of degrees $q\Phi_4^2$ and
$q^7\Phi_4^2$, respectively. These characters are of $\ell_2$-defect
zero and vanish on every $\ell_2$-singular element. However, $G$ has
another $\ell_2$-defect zero unipotent character, namely the one
labeled by ${1\hspace{6pt}2 \choose 0\hspace{6pt}3}$, of degree
$q^3\Phi_3\Phi_4^2/2$, and thus the two exceptional characters are
in the same connected component with the non-unipotent characters.
Finally, the Steinberg character is handled as above, and
$\Gamma_v(G)$ is connected.

When $(n,q)=(5,2)$, the above arguments still go through with
$\mathcal{T}_2$ replaced by a maximal torus of order
$(q^3-1)(q^2-1)$ and $\ell_2$ being $\ell(3)$ (and keep
$\mathcal{T}_1$ and $\mathcal{T}_3$). This concludes the proof.
\end{proof}

\begin{thm}
\label{spo} If $G$ is a sporadic simple group, then $\Gamma_v(G)$ is
connected.
\end{thm}

\begin{proof}
This can be checked using GAP \cite{gap}.
\end{proof}

Theorem \ref{main:D} readily follows from Theorems \ref{T280223},
\ref{lie}, \ref{spo} and the classification of finite simple groups
(the cases $\AAA_5$ and $\AAA_6$ can be easily checked and the case
of cyclic groups of prime order is a triviality).


\section{Equivalence of character-induced
metrics}\label{sec:metric}

In this section we discuss an application of Theorem \ref{main:A} to
a problem on character-induced metrics on permutations.

A metric $\bfd$ on a set $X$ is a binary function $\bfd:X\times
X\rightarrow \RR^{\geq 0}$ such that, for every $a,b,c\in X$,
\begin{itemize}
\item $\bfd(a,b)=0$ if and only if $a=b$,
\item $\bfd(a,b)=\bfd(b,a)$, and
\item $\bfd(a,b)\leq \bfd(a,c)+\bfd(c,b)$. \end{itemize}
When $X=G$ is a finite group, a metric $\bfd$ is of particular
interest when it is bi-invariant (also called $G$-invariant); that
is, \[\bfd(a,b)=\bfd(ac,bc)=\bfd(ca,cb)\] for every $a,b,c\in G$.
See \cite[Chapter 10]{Deza} for more background on the theory of
metrics on groups.

In his book \cite[Section 6D]{Diaconis}, Diaconis introduces the
\emph{matrix norm approach} as a method for constructing
(bi-invariant) metrics on finite groups.  Through this construction,
many well-known metrics on permutations, including the Hamming
distance, can be obtained. This approach relies on \emph{faithful
unitary representations}
\[\rho:G\rightarrow GL(V)\]
 and the \emph{Frobenius
norm} on matrices
\[\|M\|:=\left(\sum_{i,j} M_{ij}\overline{M}_{ij}\right)^{1/2}=Tr(MM^\ast)^{1/2}.\] (Recall that $M^\ast$ is the
conjugate transpose of a matrix $M$ and $\overline{z}$ is the
conjugate of a complex number $z$. Also, $\rho$ is unitary if
$\rho(g)\rho(g)^\ast=\rho(g)^\ast\rho(g)=I$ for every $g\in G$.) If
$\rho$ is such a representation then
\[\bfd_\rho(a,b):=\|\rho(a)-\rho(b)\|\] is a metric on $G$.
Letting $\chi$ be the character afforded by $\rho$ (and in fact,
every character can be afforded by a unitary representation), we
have
\[
\bfd_\chi(a,b):=\bfd_\rho(a,b)=\sqrt{2}(\chi(1)- Re
(\chi(ab^{-1})))^{1/2},
\]
where $Re(z)$ denotes the real part of a complex number $z$.

Clearly the distances $\bfd(1,a)$ between the identity element and
other elements of the group completely determine a bi-invariant
metric $\bfd$. Let $\mathcal{P}(\bfd)$ be the partition of $G$
determined by the equivalence relation:
\[a\thicksim b \text{ if and only if } \bfd(1,a)=\bfd(1,b).\] Two metrics
$\bfd_1$ and $\bfd_2$ are called $\mathcal{P}$-equivalent if
$\mathcal{P}(\bfd_1)=\mathcal{P}(\bfd_2)$ (see e.g. \cite{PV23}).

It is well-known that every character of the symmetric group
$\SSS_n$ is rational-valued. Therefore, if $\chi\in\Irr(\SSS_n)$
then $\mathcal{P}(\bfd_\chi)$ is determined by the relation:
\[\pi\thicksim \sigma \text{ if and only if }\chi(\pi)=\chi(\sigma).\]

Here, using Theorem~\ref{main:A}, we prove the non-equivalence of
the metrics on permutations induced from irreducible characters.

\begin{thm}\label{main:2}
Let $n\in\ZZ^{\geq 3}$. The metrics $\bfd_{\chi}$ on the
permutations in $\SSS_n$ induced from the faithful irreducible
characters $\chi$ of the group are pairwise
non-$\mathcal{P}$-equivalent.
\end{thm}

We need a few preliminary lemmas.

\begin{lem}
\label{eq1} Let $\{\la,\mu\}$ be one of the pairs
$\{(n-2,2),(n-2,1^2)\}$, $\{(2,2,1^{n-4}),(n-2,1^2)\}$,
$\{(n-2,2),(3,1^{n-3})\}$ and $\{(2,2,1^{n-4}),(3,1^{n-3})\}$. Then
$\chi^\la$ and $\chi^\mu$ induce non-$\mathcal{P}$-equivalent
metrics.
\end{lem}

\begin{proof} We provide arguments only for $\la=(n-2,2)$ and $\mu=(n-2,1^2)$ with $n\geq 4$. The other cases
are similar. Let $\sigma_1:=(1\, ... \,n-3)$ and
$\sigma_2:=\sigma_1(n-2 \,\,n-1)$. It is easy to see that $\chi^\la$
vanishes on both $\sigma_1$ and $\sigma_2$, while $\chi^\mu$ takes
values 1 on $\sigma_1$ and $-1$ on $\sigma_2$. This shows that the
partitions determined by $\chi^\la$ and $\chi^\mu$ are different, as
desired.
\end{proof}

\begin{lem}
\label{eq2} Let $\chi\in\Irr(\SSS_n)$ such that $\chi\neq
\sgn\cdot\chi$. There exist $\pi,\sigma\in\SSS_n$ of different
signature such that $\chi(\pi)=\pm \chi(\sigma)\neq 0$.
\end{lem}

\begin{proof}
Suppose that $\lambda$ is the partition of $n$ corresponding to
$\chi$. The assumption on $\chi$ implies that $\lambda$ is not
self-conjugate, or equivalently, the Young diagram $[\lambda]$ of
$\lambda$ is not symmetric. Following \cite{jk}, we use $R_{ij}$ for
the part of the \emph{rim} of $[\lambda]$ corresponding to the hook
at $(i,j)$. The length of the hook at $(i,j)$ is denoted by
$h_{ij}$.

If $\lambda\in\{(n),(1^n)\}$ then $\chi$ is the trivial or sign
character, so the result is clear ($n\geq 2$ as $\la$ is not
self-conjugate). So we may now assume that $(1,2)$ and $(2,1)$ are
both nodes of $\la$.

First we consider the case where $[\lambda]\backslash R_{11}$ is not
symmetric (in particular, $[\lambda]\backslash R_{11}$ is
non-empty). Let $\overline{\lambda}$ be the partition with Young
diagram $[\lambda]\backslash R_{11}$. By induction, there exist
$\overline{\pi}$ and $\overline{\sigma}$ in $\SSS_{n-h_{11}}$ of
different signature such that
\[
\chi^{\overline{\lambda}}(\overline{\pi})=\pm \chi^{\overline{\lambda}}(\overline{\sigma})\neq 0.
\]
Let $\tau$ be a cycle of length $n-h_{11}$ in $\SSS_{n-h_{11}}$, and
set
\[
\pi:=\tau\overline{\pi} \text{ and } \sigma:=\tau\overline{\sigma}.
\]
By the Murnaghan-Nakayama formula, we have
\[
\chi^\lambda(\pi)=\pm \chi^{\overline{\lambda}}(\overline{\pi}) \text{ and }
\chi^\lambda(\sigma)=\pm \chi^{\overline{\lambda}}(\overline{\sigma}),
\]
which implies what we wanted.

It remains to consider the case where $[\lambda]\backslash R_{11}$
is symmetric. Since $[\lambda]$ itself is symmetric but
$[\lambda]\backslash R_{11}$ is, it follows that $h_{12}\neq
h_{21}$. Without loss, assume that $h_{12}>h_{21}$, so that $h_{12}$
is the second largest hook length in $[\lambda]$ and it occurs with
multiplicity 1. Note that $[\lambda]\backslash R_{12}$ can be
obtained from $[\lambda]\backslash R_{11}$ by adding some nodes to
the first column and that $[\lambda]\backslash R_{12}$ has at least
2 nodes in the first column. It follows that $[\lambda]\backslash
R_{12}$ is not symmetric. Repeating the above arguments, we arrive
at the same conclusion.
\end{proof}

\begin{lem}
\label{eq}
Let $\chi,\psi\in\Irr(\SSS_n)$. If $\Van(\chi)=\Van(\psi)$ then
$\chi=\psi$ up to multiplying with $\sgn$.
\end{lem}

\begin{proof}
Let $\la,\mu$ be the partitions of $n$ with $\chi=\chi^\la$ and
$\psi=\chi^\mu$. We may assume that $\mu\not\in\{\la,\la'\}$ (with
$\la$ and $\la'$ being conjugated partitions).

Then by \cite[Theorem 2]{m1} $H(\la)\not= H(\mu)$ or
$H([\la]\backslash R_{11})\not=H([\mu]\backslash R_{11})$. The lemma
then follows by the proofs of \cite[Propositions 3.3.8 and
3.3.9]{m2}.
\end{proof}

\begin{proof}[Proof of Theorem \ref{main:2}] The result can be
checked for $n\leq 7$ from the known character tables, so we suppose
that $n\geq 8$. Assume that $\chi,\psi\in\Irr(\SSS_n)\backslash
\{\textbf{1}_{\SSS_n},\sgn\}$ such that
$\mathcal{P}(\bfd_\chi)=\mathcal{P}(\bfd_\psi)$, and let $\la$ and
$\mu$ be the partitions of $n$ corresponding to $\chi$ and $\mu$,
respectively.

We know that $\{\la,\mu\}$ is not one of pairs considered in Lemma
\ref{eq1}. It follows that, by Theorem~\ref{main:A}, $\chi$ and
$\psi$ have a common zero. As
$\mathcal{P}(\bfd_\chi)=\mathcal{P}(\bfd_\psi)$, we deduce that
$\Van(\chi)=\Van(\psi)$, and thus, by Lemma \ref{eq},
\[
\psi\in\{\chi, \sgn\cdot\chi\}.
\]
We therefore would be done if $\chi$ and $\sgn\cdot\chi$ produce
different partitions on $\SSS_n$; i.e. $\mathcal{P}(\chi)\neq
\mathcal{P}(\sgn\cdot\chi)$. For this, it is sufficient to show that
there exist permutations $\pi$ and $\sigma$ of different signature
such that
\[\chi(\pi)=\pm\chi(\sigma)\neq 0.\] This is done in Lemma
\ref{eq2}, and the proof is complete.
\end{proof}


\section{Relation with character degrees}\label{section:relation-character-degree}

The results we have observed suggest that the common-zero graph
$\Gamma_v(G)$ and the common-divisor graph $\Gamma(G)$ share many
similar properties. However studying $\Gamma_v(G)$ seems to be more
challenging. Both solvable and non-solvable groups with disconnected
$\Gamma(G)$ have been classified \cite{lew2,lw}. To achieve a
similar classification for $\Gamma_v$, and, in particular, to show
that $\Gamma_v(G)$ has at most three connected components for all
$G$, we believe that the following question is crucial.

\begin{que}
Let $G$ be a finite group.
Is it true that if $\Gamma(G)$ is connected then $\Gamma_v(G)$ is
connected?
\end{que}

As presented in Section \ref{sec:classes}, there are examples of
groups with irreducible characters that are linked in $\Gamma(G)$
but not in $\Gamma_v(G)$. In other words, there exist irreducible
characters that have non-coprime degrees and no common zeros. There
are also examples of irreducible characters that have common zeros
and coprime degrees, for instance the irreducible characters of
degree $3$ and $8$ in the semidirect product of $\GL_2(3)$ acting on
its natural module. Following up Lemma \ref{eq}, we wonder what
would happen if two irreducible characters have exactly the same
vanishing set.

\begin{pro}
\label{cop} Let $G$ be a finite group and let $\chi,\psi\in\Irr(G)$
such that $(\chi(1),\psi(1))=1$. Then $\Van(\chi)\not\subseteq\Van(\psi)$
and $\Van(\psi)\not\subseteq\Van(\chi)$.
In particular, if  $\Van(\chi)=\Van(\psi)$ then $(\chi(1),\psi(1))\neq 1$.
\end{pro}

\begin{proof}
By symmetry, it suffices to prove that
$\Van(\chi)\not\subseteq\Van(\psi)$. Let $\chi(1)$ is a $\pi$-number
for some set of primes $\pi$, so that  $\psi(1)$ is a $\pi'$-number.
By way of contradiction, suppose that
$\Van(\chi)\subseteq\Van(\psi)$. By \cite{mno}, there exists
$p\in\pi$ and $x\in G$ of $p$-power order such that $\chi(x)=0$. It
follows that $\psi(x)=0$. Note that $\psi$ has degree not divisible
by $p$. Therefore, Corollary 4.20 of \cite{nav} implies that
$\psi(x)\neq0$, a contradiction.
\end{proof}

In symmetric groups, more is true: two irreducible characters must
have the same degree if they have the same vanishing set, by Lemma
\ref{eq}. However, this is not the case in general, as shown by
$\SL_2(5)$ or $\PSL_2(11)$. One can find more counterexamples using
\cite{gap}, including solvable groups. These counterexamples suggest
that there should be more to say about the relationship between the
degrees of two irreducible characters having the same vanishing set.
This remains to be discovered. There is one important family of
groups among which we have found no counterexamples.

\begin{que}\label{1}
Let $G$ be a finite $p$-group. Suppose that $\chi,\psi\in\Irr(G)$
and $\Van(\chi)=\Van(\psi)$. Is it true that $\chi(1)=\psi(1)$?
\end{que}

This was communicated to one of us by J. Sangroniz years ago. We can
now show that Question \ref{1} has an affirmative answer when one of
the two characters has degree $p$. We begin with a general lemma.

\begin{lem}
\label{indp} Let $G$ be a finite group and  let $N\trianglelefteq G$. Suppose that
$|G:N|=p$ is prime. Then $\chi(x)=0$ for every $x\in G-N$ if and only if $\chi$ is
induced from $N$.
\end{lem}

\begin{proof}
The result is clear if $\chi$ is induced from $N$. Now, assume that
$\chi(x)=0$ for every $x\in G-N$.
We want to see that $\chi$ is induced from $N$.
Assume not. Then $\chi_N\in\Irr(N)$ by \cite[Corollary 6.19]{isa}. Hence
\begin{align*}
1=&[\chi,\chi]=\frac{1}{|G|}\sum_{g\in G}|\chi(g)|^2=
\frac{1}{|G|}\sum_{g\in N}|\chi(g)|^2\\
=&(1/p)\frac{1}{|N|}\sum_{g\in
N}|\chi(g)|^2=(1/p)[\chi_N,\chi_N]=1/p,
\end{align*} which is a
contradiction.
\end{proof}

As usual, if $G$ is a group $\bZ_2(G)$ is the subgroup of $G$ such
that $\bZ(G/\bZ(G))=\bZ_2(G)/\bZ(G)$.

\begin{lem}
\label{zer} Let $G$ be a finite $p$-group and $N$ be a proper normal subgroup of  $G$.
Suppose that there exists $\delta\in\Irr(N)$ such that
$\chi=\delta^G\in\Irr(G)$. Then $\Van(\chi)\cap N\neq\emptyset$.
\end{lem}

\begin{proof}
Without loss of generality, we may assume that $\chi$ is faithful.
Note that $\bZ(G)<N$. Therefore, $N/\bZ(G)\cap\bZ_2(G)/\bZ(G)>1$.
 Thus, there exists
a noncentral element $x\in N\cap\bZ_2(G)$. Now, a similar argument
as at the end of the proof of \cite[Theorem C]{ms} shows that
$\chi(x)=0$.
\end{proof}

Now,  we are ready to prove the promised result.

\begin{pro}
Let $G$ be a finite $p$-group. Suppose that $\chi,\psi\in\Irr(G)$
and $\Van(\chi)=\Van(\psi)$. If $\chi(1)=p$, then $\psi(1)=p$.
\end{pro}

\begin{proof}
We argue by induction on $|G|$. Suppose first that there exists $M$
maximal in $G$ such that $\chi_M,\psi_M\in\Irr(M)$. Then the result
follows from the inductive hypothesis.

Now, let $M$ be a maximal subgroup of $G$. The hypothesis
$\Van(\chi)=\Van(\psi)$ and Lemma~\ref{indp} imply that one of the
characters is induced from $M$ if and only if the other character is
also induced from $M$. Therefore, we may assume that for any $M$
maximal in $G$, both $\chi$ and $\psi$ are induced from $M$.

Let $\bZ(G)\leq N\trianglelefteq G$ such that $G/N$ is elementary
abelian of order $p^2$. Let $\nu\in\Irr(N)$ lying under $\chi$.  Let
$T:=I_G(\nu)$. Since $\chi(1)=p$, Clifford's correspondence implies
that $N<T$. Therefore, if $T<G$ then $T$ is a maximal subgroup of
$G$. Write $G/N=T/N\times U/N$ for some $U$ maximal in $G$. Since
$I_U(\nu)=N$, $\nu$ induces irreducibly to $\mu\in\Irr(U)$. Note
that $\nu$ lies under  $\chi$, so by comparing degrees, we have
$\chi_U=\mu$, and this is a contradiction. It follows that $\nu$ is
$G$-invariant. The previous paragraph also implies that
$G-N\subseteq\Van(\chi)=\Van(\psi)$. Now, by \cite[Problem
6.3]{isa}, $\chi$ is fully ramified with respect to $G/N$. In
particular, since $\nu$ is linear, we conclude that $\chi$ has no
zeros in $N$, whence
$$G-N=\Van(\chi)=\Van(\psi).$$

Now, let $\delta\in\Irr(N)$ lying under $\psi$. Suppose first that
$L:=I_G(\delta)$ is maximal in $G$. It follows from the Clifford
theory that $\psi(1)/\delta(1)=p$. Write $G/N=L/N\times V/N$ for
some $V$ maximal in $G$, so that $\eta=\delta^V\in\Irr(V)$ lies
under $\psi$. We conclude that $\chi_V=\eta$. By Lemma~\ref{indp},
$G-V\not\subseteq\Van(\psi)$, a contradiction.

Now, assume that $L=G$. Using \cite[Problem 6.3]{isa} again, we see
that $\psi$ is fully ramified with respect to $G/N$. Therefore,
$\psi_N=p\delta$. Since $\Van(\psi)\cap N=\emptyset$, it follows
that $\delta$ is linear, by Burnside's theorem. The result follows
in this case too.

Finally, assume that $L=N$, so that $\psi=\delta^G$. By Lemma
\ref{zer}, $\psi$ has some zero in $N$. This is the final
contradiction.
\end{proof}

As a concluding remark, we do not consider in this paper the number
of conjugacy classes on which two certain irreducible characters
vanish simultaneously, or the number of irreducible characters
sharing a certain common zero. We do think that this topic deserves
further attention.


\end{document}